\documentclass[a4paper,12pt,draft]{amsart}
\usepackage[margin=30mm]{geometry}
\usepackage{amssymb, amsmath, amsthm, amscd}

\usepackage[T1]{fontenc}
\usepackage{eucal,mathrsfs,dsfont}
\usepackage{color}

\definecolor{mno}{rgb}{0.5,0.1,0.5}

\newcommand{\R}{\mathds R}
\newcommand{\Rd}{{\mathds R^d}}
\newcommand{\Pp}{\mathds P}
\newcommand{\Ee}{\mathds E}
\newcommand{\I}{\mathds 1}
\newcommand{\Bb}{\mathscr{B}}
\newcommand{\supp}{\operatorname{supp}}

\newtheorem{theorem}{Theorem}[section]
\newtheorem{lemma}[theorem]{Lemma}
\newtheorem{proposition}[theorem]{Proposition}
\newtheorem{corollary}[theorem]{Corollary}

\theoremstyle{definition}

\newtheorem{example}[theorem]{Example}

\begin{document}

\title[Strong Feller Continuity]{\bfseries Strong Feller Continuity of Feller Processes and Semigroups}

\author{Ren\'{e} L.\ Schilling \textrm{and} Jian Wang}
\thanks{\emph{R.\ Schilling:} TU Dresden, Institut f\"{u}r Mathematische Stochastik, 01062 Dresden, Germany. \texttt{rene.schilling@tu-dresden.de}}
\thanks{\emph{J.\ Wang:}  School of Mathematics and Computer Science, Fujian Normal University, 350007, Fuzhou, P.R. China \emph{and}
TU Dresden, Institut f\"{u}r Mathematische Stochastik, 01062 Dresden, Germany.
\texttt{jianwang@fjnu.edu.cn}}

\date{}

\maketitle

\begin{abstract}
We study two equivalent characterizations of the strong Feller property for a Markov process and of the associated sub-Markovian semigroup. One is described in terms of locally uniform absolute continuity, whereas the other uses local Orlicz-ultracontractivity. These criteria generalize many existing results on strong Feller continuity and seem to be more natural for Feller processes. By establishing the estimates of the first exit time from balls, we also investigate the continuity of harmonic functions for Feller processes which enjoy the strong Feller property.

\medskip\noindent
\textbf{Keywords:} Strong Feller property; locally uniform absolute continuity; local Orlicz-ultracontractivity; Feller processes; harmonic functions; first exit times.

\medskip\noindent
\textbf{MSC 2010:} 60J35; 47D07; 60G51; 60J25; 60J75; 47G05.
\end{abstract}

\section{Introduction}\label{intro}
Let $(\{X_t\}_{t\geq 0}, \Pp^x)$ be a Markov process on $\R^d$ with transition probability function $P_t(x,dy)$. For any $f\in B_b(\R^d)$, the set of bounded Borel measurable functions on $\R^d$, define
$$
    T_tf(x)=\int f(y)P_t(x,dy),\quad x\in \R^d.
$$
Then, $\{T_t\}_{t\ge 0}$ is a sub-Markovian semigroup on $B_b(\R^d)$, i.e.\ $T_t:B_b(\Rd)\to B_b(\Rd)$ and if $u\in B_b(\R^d)$ with $0\le u\le 1$ then $0\le T_tu\le 1$. Write $C_b(\R^d)$ (resp. $C_\infty(\R^d)$) for the set of bounded continuous functions (resp. continuous functions vanishing at infinity). The semigroup $\{T_t\}_{t\geq 0}$ has the
\begin{description}
    \item[\emph{Feller property}] if $T_t f\in C_\infty(\R^d)$ for all $f\in C_\infty(\R^d)$ and all $t\geq 0$.
    \item[$C_b$-\emph{Feller property}] if $T_t f\in C_b(\R^d)$ for all $f\in C_b(\R^d)$ and all $t\geq 0$.
    \item[\emph{strong Feller property}] if $T_t f\in C_b(\R^d)$ for all $f\in B_b(\R^d)$ and all $t\geq 0$.
\end{description}
A \emph{Feller semigroup} is a sub-Markovian semigroup which has the Feller property and which is on $C_\infty(\R^d)$ strongly continuous: $\lim_{t\to 0}\|T_t f - f\|_\infty =0$ for all $f\in C_\infty(\R^d)$. A \emph{Feller process} is a Markov process where the associated semigroup is a Feller semigroup.

In this paper we investigate the strong Feller property of a sub-Markovian semigroup $\{T_t\}_{t\ge 0}$. It is obvious that the strong Feller property implies the $C_b$-Feller property---but it does not necessarily entail the Feller property. Some criteria ensuring that a sub-Markovian semigroup has the $C_b$-Feller property are known, see e.g.\ \cite[Theorem 3.2]{S1}; this condition is by no means sharp and as far as we know no sharp condition is currently known. Here, we will establish two equivalent criteria for the strong Feller property which take different forms:  one is based on locally uniform absolute continuity, see Theorem \ref{thma}; the other one is based on local Orlicz-ultracontractivity, see Theorem \ref{ou1}.

For Feller processes, the above conditions can be simply expressed by locally uniform absolute continuity and local  Orlicz-ultracontractivity with respect to Lebesgue measure. For L\'evy processes---these are the spatially homogeneous Feller processes---this is not only equivalent to Hawkes's well-known result \cite{HAW} on the existence of transition densities, but also yields a new characterization for the strong Feller property of L\'{e}vy processes, see Corollary \ref{newfeller}. Strong Feller continuity is an interesting property in its own right, and is also needed in many applications, e.g.\ for the equivalence of transition probabilities, for ergodic properties, etc. We will also point out the relationship between semigroups (resp.\ resolvents) enjoying the strong Feller property and the continuity of harmonic functions for general Feller processes.

\section{General results: Sub-Markovian semigroups}\label{submark}
Let $\{T_t\}_{t\ge 0}$ be a sub-Markovian semigroup on $B_b(\R^d)$ which has the $C_b$-Feller property. At this point we do not assume strong continuity of the semigroup. Recall that  $\{T_t\}_{t\ge 0}$ is said to have the strong Feller property if $T_t : B_b(\R^d)\to C_b(\R^d)$ for all $t>0$. A general exposition of the strong Feller property can be found in \cite{Gs}. Denote by $P_t(x,\cdot)$ the kernel representing $T_t$, i.e.\ $P_t(x,A):=T_t\I_{A}(x)$ for all $t>0$, $x\in\R^d$ and $A\in {\Bb(\R^d)}$.

\subsection{Criterion I: Locally uniform absolute continuity}

\begin{theorem}\label{thma}
    Let $\{T_t\}_{t\ge 0}$ be a sub-Markovian semigroup on $B_b(\R^d)$. Then the following properties are equivalent:
    \begin{itemize}
    \item[(a)]
        $\{T_t\}_{t\ge 0}$ has the strong Feller property;
    \item[(b)]
        $\{T_t\}_{t\ge 0}$ has the $C_b$-Feller property and for every $t>0$ there exists some positive Radon measure $\mu_t$ on $\R^d$ such that the family of measures $(P_t(x,dy))_{x\in\R^d}$ is locally uniformly absolutely continuous with respect to $\mu_t$, i.e.\ for any compact set $K\subset\R^d$ it holds that
        $$
            \lim_{\delta\to 0}\;\sup_{A\in {\Bb(\R^d)},\, \mu_t(A)\le \delta}\;\sup_{z\in K} P_t(z,A) = 0.
        $$
    \item[(c)]
        $\{T_t\}_{t\ge 0}$ has the $C_b$-Feller property and for every $t>0$ there exists a probability measure $\mu_t$ on $\R^d$ such that the family of measures $(P_t(x,dy))_{x\in\R^d}$ is locally uniformly absolutely continuous with respect to $\mu_t$.
    \end{itemize}
    If, in addition, the semigroup $\{T_t\}_{t\ge0}$ is such that for every $f\in C_c^\infty(\R^d)$, $T_tf$ converges locally uniformly to $f$ as $t$ tends zero, then all statements above are also equivalent to
\begin{itemize}
    \item[(d)]$
        \{T_t\}_{t\ge 0}$ has the $C_b$-Feller property and there exists a probability measure $\mu$ on $\R^d$ such that for every $t>0$ the family of measures $(P_t(x,dy))_{x\in\R^d}$ is locally uniformly absolutely continuous with respect to $\mu$.
\end{itemize}
\end{theorem}

Usually it is straightforward to check the $C_b$-Feller property. For example, according to \cite[Theorem 3.2; (ii)$\Rightarrow$ (iii)]{S1}, $\{T_t\}_{t\ge 0}$ has the $C_b$-Feller property if and only if for each $t\ge0$, $T_t$ maps $C_c^\infty(\R^d)$ into $C_b(\R^d)$ and $T_t1\in C_b(\R^d)$. By Theorem \ref{thma} and its proof, we have the following natural way to prove the strong Feller property once we have the $C_b$-Feller property.

\begin{corollary}\label{thma1}
    Let $\{T_t\}_{t\ge0}$ be a sub-Markovian semigroup which has the $C_b$-Feller property. If for every $t>0$, the kernel $P_t(x,\cdot)$ representing $T_t$ has a density with respect to some positive Radon measure $\mu$, i.e.\ $$
        P_t(x,dy)=p_t(x,y)\mu(dy),
    $$
    and if, moreover, the density function $p_t(x,y)$ is locally bounded on $\R^d\times\R^d$, then $\{T_t\}_{t\ge0}$ has the strong Feller property.
\end{corollary}

\bigskip

To prove Theorem \ref{thma}, we first derive some properties of semigroups enjoying the strong Feller property which are also interesting for themselves.
\begin{proposition}\label{strong1}
    Let $\{T_t\}_{t\ge 0}$ be a sub-Markovian semigroup with the strong Feller property. Then, for every $t>0$, there exists a probability measure $\mu_t$ such that $T_t$ is well defined on $L^\infty(\mu_t)$ and $T_t:L^\infty(\mu_t)\to  C(\R^d)$ is a compact operator. Consequently, for all compact sets $K\subset \R^d$,
\begin{equation*}
    \lim _{\delta\to 0}\;
        \sup_{|x-y|\le \delta,\: x,y\in K}\;
        \sup_{\|u\|_{L^\infty(\mu_t)}\le 1}
        |T_tu(x)-T_tu(y)|
        =0.
\end{equation*}
\end{proposition}

\begin{proof}
    For every $t>0$ denote by $P_t(x,\cdot)$ the kernel representing $T_t$. Choose a continuous function $w > 0$ on $\R^d$ such that $\int w(x)\,dx=1$. For every $t>0$ define
    $$
        \mu_t(\cdot)
        =\frac{\int w(x)P_t(x,\cdot)\,dx}{\int w(x)P_t(x,\R^d)\,dx}.
    $$
    If $N\in \Bb(\R^d)$ satisfies that $\mu_t(N)=0$, then we have $P_t(x,N)=0$ for Lebesgue a.e.\ $x$. Since by the strong Feller property the function $x\mapsto T_t\I_N(x) = P_t(x,N)$ is continuous, we find that $P_t(x,N)=0$ for all $x\in\R^d$. This means that $P_t(x,\cdot)$ is absolutely continuous with respect to $\mu_t$ and has a Radon-Nikod\'ym density $p_t(x,y)$. Now, for any $u\in L^\infty(\mu_t)$, set $N=\{x\in\R^d\::\: |u(x)| > \|u\|_{L^\infty(\mu_t)}\}\in \Bb(\R^d)$. Then, $\mu_t(N)=0$. Set $\tilde{u}:= u\cdot \I_{N^c}$ and note that $\tilde u$ is bounded and Borel measurable. Naturally, we define
    $$
        T_tu(x)
        :=\int u(y)P_t(x,dy)
        =\int u(y)p_t(x,y)\,\mu_t(dy)
        \quad\text{for\ \ } u\in L^\infty(\mu_t).
    $$
    Clearly, It holds that $T_tu(x)=T_t\tilde{u}(x)$ for every $x\in\R^d$. That is, $T_t$ is well defined on $L^\infty(\mu_t)$ and $T_t(L^\infty(\mu_t))\subset C_b(\R^d)$, due to the strong Feller property.

    To prove our second assertion, it is sufficient to show that the image of the unit ball $U:=\{u\in L^\infty(\mu_t):\|u\|_{L^\infty(\mu_t)}\le1\}$ under $T_t$ is sequentially compact in $C(\R^d)$ with local uniform topology. Since $C(\R^d)$ is a metric space, this implies compactness.

    Noticing that $L^\infty(\mu_t)=\big[L^1(\mu_t)\big]^*$, the Banach-Alaoglu theorem
    tells us that $U$ is weak-$*$, i.e.\ $\sigma(L^\infty(\mu_t),L^1(\mu_t))$, compact. This entails that every sequence $\{u_j\}_{j\in \mathds{N}}\subset U$ has a weak-$*$ convergent subsequence $\{u_{j(k)}\}_{k\in \mathds{N}}$ and, for a suitable $u\in L^\infty(\mu_t)$, the limit
    $$
        \lim_{k\to \infty}\int u_{j(k)}(y)\phi(y)\,\mu_t(dy)
        =\int u(y)\phi(y)\,\mu_t(dy)
        \quad\text{for every\ \ } \phi\in L^1(\mu_t)
    $$
    exists. Therefore, the sequence
    $$
        T_t u_{j(k)}(x)=\int  u_{j(k)}(y)\,P_t(x,dy)
        = \int  u_{j(k)}(y)p_t(x,y)\,\mu_t(dx)
    $$
    converges pointwise for every $x\in\R^d$. Moreover, for $k, l, m\in \mathds{N}$ with $k, l\ge m$,
    $$
        |T_{2t}u_{j(k)}-T_{2t}u_{j(l)}|
        \le T_t\big|T_{t}u_{j(k)}-T_{t}u_{j(l)}\big|
        \le T_t\Big(\sup_{k,l\ge m} \big|T_{t}u_{j(k)}-T_{t}u_{j(l)}\big|\Big).
    $$
    Note that $h_m:=\sup_{k,l\ge m}\big|T_{t}u_{j(k)}-T_{t}u_{j(l)}\big|$ decreases to 0 as $m \to \infty$, and so does the sequence $(T_th_m)_{m\in \mathds{N}}$, thanks to the dominated convergence theorem. By the strong Feller property, the functions $(T_th_m)_{m\in \mathds{N}}$ are continuous, and Dini's theorem shows locally uniform convergence. This means that $\{T_tu_{j(k)}\}_{k\in \mathds{N}}$ is a Cauchy sequence in $C(\R^d)$ under locally uniform convergence, and so the limit $T_tu=\lim_{k\to \infty}T_tu_{j(k)}$ defines a continuous function. The proof is complete.
\end{proof}

Since $B_b(\R^d)\subset L^\infty(\mu_t)$, the following conclusion is an immediate consequence of Proposition \ref{strong1}.

\begin{corollary}\label{compbu}
    Let $\{T_t\}_{t\ge0}$ be a strong Feller semigroup. Then $T_t:B_b(\R^d)\to  C(\R^d)$ is for every $t>0$ a compact operator. In particular, for all compact sets $K\subset \R^d$,
    \begin{equation}\label{com}
        \lim _{\delta\to 0}\:
        \sup_{|x-y|\le \delta,\,\, x,y\in K}\:
        \sup_{\|u\|_\infty\le 1}
        |T_tu(x)-T_tu(y)|=0.
    \end{equation}
\end{corollary}

According to \cite[Corollary 1.3]{JSLP}, we know that if $\{T_t\}_{t\ge 0}$ is a sub-Markovian semigroup with the strong Feller property, all operators $T_t:B_b(\R^d)\to  C(\R^d)$ are bounded. The main point of  Corollary \ref{compbu} is to show that the strong Feller continuity guarantees that the bounded operator $T_t$ are actually compact. A close inspection of the proof of Proposition \ref{strong1} reveals that the image of the unit ball of $L^\infty(\mu_t)$ is just that of $B_b(\R^d)$ under the operator $T_t$, which is a compact subset of $C(\R^d)$. The following example shows that this property is not fulfilled if we only assume the $C_b$-Feller property of the semigroup.

\begin{example}
    Consider the shift semigroup,
    $$
        \tau_t:B_b(\R^d)\to  B_b(\R^d),
        \quad \tau_tu(x):=u(x-t), \quad x\in\R^d,\; t\ge0.
    $$
    Obviously, $\{\tau_t\}_{t\ge0}$ enjoys the $C_b$-Feller property but it does not have the strong Feller property. The functions
    $$
        u_j(x):=\I_{[-\pi,\pi]}(x)\,\frac{1}{\sqrt{\pi}}\sin (jx),\quad j\in \mathds{N},
    $$
    are clearly continuous and uniformly bounded. Using the $L^2(dx)$-orthonormality of this family, we see
    $$
        \int |u_j(x)-u_k(x)|^2\,dx
        =\int_{-\pi}^\pi\big(u_j^2(x)+u_k^2(x)\big) dx=2,
    $$
    which means that $\{u_j\}_{j\in \mathds{N}}$ does not converge in $L^2(dx)$ and cannot have any pointwise convergent subsequence.

    Since the shifted sequence $\{u_j(\cdot-t)\}_{j\in \mathds{N}}$ inherits this property, we obtain that $\tau_t(\{u_j\}_{j\in \mathds{N}})$ is not even weakly compact in $C(\R^d)$. (Under the locally compact topology on $C(\R^d)$, the weak convergence is just pointwise convergence.)
\end{example}

\bigskip

The next proposition is essentially taken from \cite[Proposition 2.10]{BLIE}. For the sake of completeness, we include its proof.

\begin{proposition}\label{strong2}
     Let $\{T_t\}_{t\ge0}$ be a sub-Markovian semigroup with the strong Feller property and such that for every $u\in C_c^\infty(\R^d)$, $T_tu$ converges locally uniformly to $u$ as $t$ tends zero. Then the map $(s,x,u)\mapsto T_su(x)$ is jointly continuous for all $s\ge 0$, $x\in\R^d$ and $u\in B_b(\R^d)$.
\end{proposition}

\begin{proof}
    First we claim that for any compact set $K\subset \R^d$ and every $\varepsilon>0$ and $g\in C_b(\R^d)$, there is some $\delta:=\delta(K, g, \varepsilon)>0$ such that
    \begin{equation}\label{poro1}
        \sup_{t\le \delta}\|(T_tg-g)\I_K\|_\infty \le \varepsilon.
    \end{equation}
    Indeed, choose a cut-off function $\chi\in C_c^\infty(\R^d)$ with $\I_K\le \chi\le 1.$ Without loss of generality, we may assume that $0\le g\le 1$. Because of the continuity of $T_t$ at $t=0$ on $C_c^\infty(\R^d)$, there is some $\delta>0$ such that
    $$
        [g\chi-T_t(g\chi)]\,\I_K\le \varepsilon,\quad
        [(1-g)\chi-T_t((1-g)\chi)]\I_K\le \varepsilon\quad
        \text{for all\ \ } t\le \delta.
    $$
    So,
    $$
        (g-\varepsilon)\I_K=(g\chi-\varepsilon)\I_K\le \I_KT_t(g\chi)\le \I_KT_tg
    $$
    and
    \begin{align*}
        \I_KT_tg
        = \I_K[T_t1-T_t(1-g)]
        &\le \I_K[T_t1-T_t((1-g)\chi)]\\
        &\le \I_K[T_t1-(1-g)\chi+\varepsilon]\le \I_K[g+\varepsilon].
    \end{align*}
    This proves \eqref{poro1}.

    \smallskip
    Let $u\in B_b(\R^d)$ with $0\le u\le 1$, $t>0$ and $\varepsilon>0$. Given a compact set $K\subset \R^d$, choose $\chi\in C_c^\infty(\R^d)$ such that $\I_{K}\le \chi\le 1$. By \eqref{poro1}, there exists some $0<\delta_1<t$ such that for every $0<s\le \delta_1$,
    $$
        \|\I_K T_s(1-\chi)\|_\infty <\varepsilon.
    $$
    Let $t_0=t-\delta_1$ and $g=T_{t_0}u$. Then $g\in C_b(\R^d)$ and using \eqref{poro1} once again, we find some $0<\delta<\delta_1$ such that for every $0<s<\delta$,
    $$
        \|\I_{\supp(\chi)} (T_sg-g)\|_\infty <\varepsilon.
    $$
    Consider now $t\le s<t+\delta$. Since $1\le 1-\chi+\I_{\supp(\chi)}$ and $|T_{s-t}g-g|\le \|T_{s-t}g\|_\infty + \| g\|_\infty \le 2$, we have, on $K$,
    \begin{align*}
        |T_su-T_tu|
        &=|T_{t-t_0}(T_{s-t}g-g)|\\
        &\le 2\,T_{t-t_0}(1-\chi)+T_{t-t_0}(\I_{\supp(\chi)}|T_{s-t}g-g|)\\
        &\le 2\varepsilon+T_{t-t_0}\varepsilon\\
        &\le 3\varepsilon.
    \end{align*}
    Similarly, if $t-\delta<s\le t$, then
    $$
        |T_su-T_tu|=|T_{t-t_0}(g-T_{s-t}g)|\le 3\varepsilon\quad\text{on\ \ } K.
    $$
    Thus, $T_su$ converges to $T_tu$ uniformly on $K$ as $s$ tends to $t$.

    According to \eqref{com}, for fixed $s>0$, $T_su(y)$ converges to $T_su(x)$ uniformly in $x,y\in K$ and $\|u\|_\infty\leq 1$ as $y$ tends to $x$. The required assertion follows easily from the following fact that for any two $u, v\in B_b(\R^d)$, $s, t>0$ and $x, y\in K$,
    \begin{align}\label{poro2}
    |T_su(x)-&T_tv(y)|\notag\\
    &\le |T_su(x)-T_su(y)|+|T_su(y)-T_tu(y)|+|T_tu(y)-T_tv(y)|\\
    &\le |T_su(x)-T_su(y)|+|T_su(y)-T_tu(y)|+\|u-v\|_\infty .\notag\qedhere
    \end{align}
\end{proof}

Let $\{T_t\}_{t\ge0}$ be a sub-Markovian semigroup with the $C_b$-Feller property and such that $T_tu$ converges locally uniformly to $u$ as $t$ tends zero for every $u\in C_c^\infty(\R^d)$. Then, by \eqref{poro2} and the fact that $C_c^\infty(\R^d)$ is dense in $C_\infty(\R^d)$, one could conclude that the map $(s,x,u)\mapsto T_su(x)$ is jointly continuous for all $s\ge 0$, $x\in\R^d$ and $u\in C_b(\R^d)$, see e.g.\ \cite[Section 2.2; Theorem 1]{CHZH}. Therefore, the argument above shows that, if $\{T_t\}_{t\ge0}$ has the strong Feller property, then $T_tu$ converges locally uniformly to $u$ as $t$ tends zero for every $u\in C_b(\R^d)$. Combining Propositions \ref{strong1} and \ref{strong2}, we obtain

\begin{corollary}\label{strong22}
    Let $\{T_t\}_{t\ge0}$ be a sub-Markovian semigroup with the strong Feller property and such that for every $u\in C_c^\infty(\R^d)$, $T_tu$ converges locally uniformly to $u$ as $t$ tends zero. Then there exists a probability measure $\mu$ such that for every $t>0$, $T_t$ is well defined on $L^\infty(\mu)$ and $T_t:L^\infty(\mu)\to  C(\R^d)$ is a compact operator.
\end{corollary}

\begin{proof}
    Pick a continuous function $w > 0$ on $\R^d$ such that $\int w(x)\,dx=1$. Define
    $$
        \mu(\cdot)
        =\frac{\int_0^\infty e^{-t}\,dt\int w(x)P_t(x,\cdot)\,dx}{\int_0^\infty e^{-t}dt\int w(x)P_t(x,\R^d)\,dx}.
    $$
    If $N\in \Bb(\R^d)$ satisfies $\mu(N)=0$, then we have $P_t(x,N)=0$ for Lebesgue a.e.\ $x$ and $t$. Since by Proposition \ref{strong2} the function $(t,x)\mapsto P_t(x,N)$ is continuous, we find that $P_t(x,N)=0$ for all $t>0$ and $x\in\R^d$. This means that $P_t(x,\cdot)$ is absolutely continuous with respect to $\mu$. Then, the desired assertion follows from the argument of Proposition \ref{strong1}.
\end{proof}

\bigskip

We can now turn to the proof of Theorem \ref{thma}.

\begin{proof}[Proof of Theorem $\ref{thma}$]
$(b)\Rightarrow(a)$
    Let $u\in B_b(\R^d)$ and let $\chi_k\in C_c^\infty(\R^d)$ be a sequence of cut-off functions with $\I_{B_k(0)}\le \chi_k\le 1$, $k\in \mathds{N}$. By the monotone convergence theorem, $\sup_{k\in \mathds{N}}T_t\chi_k=T_t1$. Since $\{T_t\}_{t\ge0}$ has the $C_b$-Feller property, $T_t\chi_k$ and $T_t1$ are continuous. Dini's theorem shows that this convergence is locally uniform. In particular, for every compact set $L\subset \R^d$ and every $\varepsilon>0$ there exists an index $k_0\in\mathds N$ such that for $\chi = \chi_{k_0}$
    \begin{equation}\label{sch1}
        \|\I_LT_t(1-\chi)\|_\infty\le\frac{ \varepsilon}{2\,(2\|u\|_\infty+1)}.
    \end{equation}

   Next, we suppose that $\supp(\chi)\subset B_{k_1}(0)$ for some $k_1\in \mathds{N}$. The function $u\chi$ is also bounded and has compact support, thus $u\chi\in L^1(\mu_t)$. Since $C_c^\infty(\R^d)\subset L^1(\mu_t)$ is dense, we can choose a sequence $\{u_j\}_{j\in \mathds{N}}\subset C_c^\infty(\R^d)$ such that $u_j$ converges to $u\chi$ almost everywhere with respect to $\mu_t$. Without loss of generality we may assume that $\sup_{j\in \mathds{N}}\|u_j\|_\infty\le \|u\|_\infty+1$ and $\bigcup_{j\in \mathds{N}} \supp(u_j)\subset B_{k_1+1}(0)$. That is, $u\chi$ and $(u_j)_{j\in \in \mathds{N}}$ can be temporarily seen as functions only defined on $B_{k_1+1}(0)$. Since $\mu_t(B_{k_1+1}(0))<\infty$, Egorov's theorem applies
   and for any $\delta>0$, there exists some Borel set $N=N(\delta)\subset B_{k_1+1}(0)$ such that $\mu_t(N)\le \delta$ and $\{u_j\}_{j\in \mathds{N}}$ converges to $u\chi$ uniformly on $B_{k_1+1}(0)\setminus N$. Therefore, $\{u_j\}_{j\in \mathds{N}}$ converges to $u\chi$ uniformly on $N^c,$ since $\{u_j\}_{j\in \mathds{N}}$ and $u\chi$ are zero on $B_{k_1+1}(0)^c$.

    For any $j\ge1$ and $z\in L$,
    \begin{equation}\label{proof10}\begin{aligned}
        \big|T_t(&u\chi)(z)-T_tu_j(z)\big|\\
        &\le \bigg|\int_{N^c}(u\chi)(y)P_t(z,dy)-\int_{N^c}u_j(y)P_t(z,dy)\bigg|\\
        &\quad+\bigg|\int_{N}(u\chi)(y)P_t(z,dy)\bigg|+\bigg|\int_{N}u_j(y)P_t(z,dy)\bigg|\\
        &\le \sup_{y\in N^c}|(u\chi)(y)-u_j(y)|\,\sup_{z\in L}P_t(z,N^c)+(2\|u\|_\infty+1)\,\sup_{z\in L}P_t(z,N)\\
        &\le\sup_{y\in N^c}|(u\chi)(y)-u_j(y)|+(2\|u\|_\infty+1)\sup_{z\in L}P_t(z,N).
    \end{aligned}\end{equation}
    Letting $j\to \infty$ and then $\delta\to 0$, the locally uniform absolute continuity of $P_t(x,\cdot)$ with respect to $\mu_t$ yields that
    $$
        \big\|\I_{L}\big(T_t(u\chi)-T_tu_j\big)\big\|_\infty \to 0\quad \text{as\ \ } j\to \infty.
    $$

    The continuity of the function $T_tu$ easily follows from the arguments above. Indeed, for any $x\in\R^d$ and $\varepsilon, \eta>0$, take a compact set $L$ such that $B(x,\eta)\subset L$. Choose $\chi$ as above and $j_0=j_0(\varepsilon)$ so large that
    $$
        \sup_{y\in B(x,\eta)} \left| T_t(u\chi)(y)-T_tu_{j_0}(y) \right|\le \frac{\varepsilon}{6}.
    $$
    On the other hand, since $u_{j_0}\in C_c^\infty(\R^d)$, we have $T_tu_{j_0}\in C(\R^d)$. Thus, there exists a constant $\eta_0=\eta_0(x)\in (0, \eta)$ such that
    \begin{equation*}
        \sup_{y\in B(x, \eta_0)}\left| T_tu_{j_0}(y)-T_tu_{j_0}(x) \right|
        \le \frac{\varepsilon}{6}.
    \end{equation*}
    Therefore, for every $y\in B(x,\eta_0)$, we can use \eqref{sch1} to see
    \begin{align*}
    |T_t&u(x)-T_tu(y)|\\
    &\le |T_t(u\chi)(x)-T_t(u\chi)(y)|+|T_t(u(1-\chi))(x)|+|T_t(u(1-\chi))(y)|\\
    &\le|T_t(u\chi)(x)-T_t(u\chi)(y)|+2\|u\|_\infty\|\I_LT_t(1-\chi)\|_\infty\\
    &\le |T_t(u\chi)(x)-T_t(u\chi)(y)|+\frac{{\varepsilon}}{2}\\
    &\le\big|T_t(u\chi)(x)-T_tu_{j_0}(x)\big|+\big| T_tu_{j_0}(y)-T_tu_{j_0}(x) \big| + |T_t(u\chi)(y)-T_tu_{j_0}(y)\big|+\frac{{\varepsilon}}{2}\\
    &\le\frac{\varepsilon}{6}+\frac{\varepsilon}{6}+\frac{\varepsilon}{6}+ \frac{\varepsilon}{2}
    = \varepsilon,
    \end{align*}
    which proves (a).

    \medskip\noindent $(a)\Rightarrow(c)$
  According to Proposition \ref{strong1}, for every $t>0$, there exists a probability measure $\mu_t$ such that $T_t: L^\infty(\mu_t)\to C(\R^d)$ is well defined. In particular, for any $A\in \Bb(\R^d)$ with $\mu_t(A)=0$, $T_t \I_{A}=0$, since $\I_{A}$ is in the same equivalent class of the zero vector in the space $L^\infty(\mu_t)$.

    Let $L\subset\R^d$ be a compact set and $n\ge1$. By the very definition of the supremum, for every $n\in\mathds N$ there exists a set $A_n\in \Bb(\R^d)$ with $\mu_t(A_n)\le 2^{-n}$ such that
    $$
        \sup_{ A\in \Bb(\R^d),\, \mu_t(A)\le 2^{-n}}\:\sup_{x\in L}P_t(x,A)
        \le \sup_{x\in L}P_t(x,A_n)+1/n.
    $$
    Set $\tilde A_n := \bigcup_{k \geq n} A_k$. Then the sequence $\tilde A_n$ is decreasing, $\tilde A = \bigcap_{n\in\mathds N} \tilde A_n$, $\mu_t(\tilde A_n) \leq 2^{-n+1}$, $\mu_t(\tilde A)=0$ and $\I_{\tilde A_n}\to  \I_{\tilde A}$. Since $T_t \I_{\tilde A_n}$ is continuous and since $T_t \I_{\tilde A_n}$ decreases to $T_t\I_{\tilde A}=0$ as $n\to\infty$, Dini's theorem applies and shows that
    $$
        \sup_{x\in L}P_t(x,A_n)
        \leq \sup_{x\in L}P_t(x,\tilde A_n)
        = \sup_{x\in L}T_t \I_{\tilde A_n}(x)
        \to 0\quad\text{as\ \  }n\to \infty.
    $$
    This proves (c).

    \medskip

    \noindent $(c)\Rightarrow(b)$ and $(d)\Rightarrow(c)$ This is clear.

    Assume now that, in addition, $T_t f\to f$, $f\in C_c(\R^d)$, locally uniformly as $t\to 0$.

    \medskip

    \noindent
    $(a)\Rightarrow(d)$ This follows form Corollary \ref{strong22} and the proof of $(a)\Rightarrow(c)$.
\end{proof}

\begin{proof}[Proof of Corollary \ref{thma1}]
    The proof of $(b)\Rightarrow(a)$ of Theorem \ref{thma} uses locally uniform absolute continuity of $(P_t(x,\cdot))_{x\in\R^d}$ only for \eqref{proof10}. Note that the set $N$ here is a bounded set and that we can assume that $N\subset K$ for some compact set $K\subset\R^d$. Fix some compact set $L$ and pick any $x\in L$. Under the conditions of Corollary \ref{thma1} we get
    \begin{align*}
        \sup_{z\in L} P_t(z,N)
        &=\int_ Np_t(z,y)\mu(dy)\\
        &\le\int_ Np_t(x,y)\mu(dy)+\sup_{z\in L} \int_ N|p_t(z,y)-p_t(z,y)|\mu(dy)\\
        &\le P_t(x,N)+2\,\mu(N)\sup_{(z,y)\in L\times K}p_t(z,y).
    \end{align*}
    This inequality and the argument of the step $(b)\Rightarrow(a)$ in the proof of Theorem \ref{thma} prove the strong Feller property of $\{T_t\}_{t\ge0}$.
\end{proof}

\subsection{Criterion II: Local Orlicz-ultracontractivity}

Let us recall some facts about Orlicz space from \cite{R}. A function $\Phi:\R\to  [0,\infty]$ is a \emph{Young function} if it is convex, even, and satisfies $\Phi(0)=0$ and $\lim_{x\to \infty}\Phi(x)=\infty$. Given a Young function and a Radon measure $\mu$ on $\R^d$, we define the \emph{Orlicz space} as
$$
    \mathds{L}^\Phi(\mu)
    =\left\{f\::\: \R^d\to \R \text{\ \ measurable and\ \ } \int \Phi(\alpha f)\,d\mu < \infty \text{\ \ for some \ \ }\alpha>0\right\}.
$$
The set $\mathds{L}^\Phi(\mu)$ is a linear space. If $\Phi(x)=|x|^p$, $p\geq 1$, then $\mathds{L}^\Phi(\mu)$ coincides with the usual Lebesgue space $\mathds{L}^p(\mu)$. There exist two equivalent norms turn $\mathds{L}^\Phi(\mu)$ into a Banach space, the \emph{Luxemburg norm}
$$
    \|f\|_{(\Phi)}
    =\inf\left\{\lambda>0\::\: \int \Phi(f/\lambda)\,d\mu \le 1\right\}
$$
and the \emph{Orlicz norm}
$$
    \|f\|_{\Phi}
    =\sup\left\{ \int |fg|\,d\mu \::\: \int \Phi_c(g)\,d\mu \le 1\right\},
$$
where $\Phi_c$ is the Legendre transform of $\Phi$, i.e.
$$
    \Phi_c(y):=\sup_{x\geq 0} \big(x|y|-\Phi(x)\big).
$$
For any Young function $\Phi$, we have
$$
    \|f\|_{(\Phi)}
    \le \|f\|_{\Phi}
    \le 2\,\|f\|_{(\Phi)}.
$$
On the other hand, the following extension of H\"{o}lder's inequality holds
$$
    \int |fg|\,d\mu
    \le 2\,\|f\|_{(\Phi)}\|g\|_{(\Phi_c)} \quad\text{for all\ \ }f\in\mathds{L}^\Phi(\mu)\text{\ \ and\ \ } g\in\mathds{L}^{\Phi_c}(\mu).
$$

We can now characterize the strong Feller property in terms of local Orlicz-ultracontractivity.

\begin{theorem}\label{ou1}
    Let $\{T_t\}_{t\ge0}$ be a sub-Markovian semigroup with the $C_b$-Feller property. If for every $t>0$, there exist a Radon measure $\mu_t$ and some Young function $\Phi_t:\R\to  \R^+$ with $\Phi_t(x)=0$ iff $x=0$, such that for all compact sets $K\subset\R^d$ and $u\in C_c^\infty(\R^d)$
    \begin{equation}\label{cou1}
        \|\I_KT_tu\|_\infty\le C(K,t)\,\|u\|_{\Phi_t},
    \end{equation}
    then $\{T_t\}_{t>0}$ is a strong Feller semigroup.

    Conversely, let $\{T_t\}_{t\ge0}$ be a sub-Markovian semigroup with the strong Feller property. Then, for every $t>0$, there exist a probability measure $\mu_t$ and a Young function $\Phi_t:\R\to  \R^+$ with $\lim_{x\to \infty} \Phi_t(x)/x=\infty$ such that for all compact sets $K\subset\R^d$ and all $u\in C_c^\infty(\R^d)$
    \begin{equation}\label{cou2}
        \|\I_KT_tu\|_{\textrm{L}^\infty(\mu_t)}
        \le C(K,t)\,\|u\|_{\Phi_t}.
    \end{equation}
    If, in addition, one of the following conditions applies, \eqref{cou1} holds for all compact sets $K\subset\R^d$ and all $u\in C_c^\infty(\R^d)$
    \begin{itemize}
        \item[(1)] $\{T_t\}_{t\ge0}$ is weakly irreducible, i.e.\ for each $t>0$ and every non-empty open set $A\subset\R^d$, there exists some $x\in\R^d$ such that $P_t(x,A)>0$.
        \item[(2)] For every $f\in C_c^\infty(\R^d)$, $T_tf$ converges locally uniformly to $f$ as $t$ tends zero. \end{itemize}
\end{theorem}

Note that irreducibility usually means that for all $t>0$, all open sets $A\in\Bb(\R^d)$ and \emph{all} $x\in\R^d$ the transition probability $P_t(x,A)>0$. This is stronger than our notion of weak irreducibility from Theorem \ref{ou1}. Theorem \ref{ou1} extends \cite[Theorem 8.9]{HOH} where the following condition has been used:
$$
    \|T_tu\|_\infty\le c_t\|u\|_2
    \quad\text{for\ \ } t>0 \text{\ \ and\ \ }u\in C_c^\infty(\R^d).
$$

We begin with some properties of Orlicz spaces, which will are used in the proof of Theorem \ref{ou1}.

\begin{lemma}\label{cou1s}
    Let $(\mathds{L}^\Phi(\mu), \|\cdot\|_\Phi)$ be an Orlicz space associated with a Radon measure $\mu$ and some Young function $\Phi:\R\to  \R^+$ with $\Phi(x)=0$ if and only if $x=0$. Then we have
    \begin{enumerate}
        \item
        Every bounded measurable function with compact support belongs to the space $(\mathds{L}^\Phi(\mu), \|\cdot\|_\Phi)$.  For any $A\in\Bb(\R^d)$ with $\mu(A)<\infty$, $\I_A\in \mathds{L}^\Phi(\mu)$ and $$
            \|\I_A\|_{(\Phi)}=\big(\Phi (1/\mu(A))\big)^{-1};
        $$
    \item
        If $(f_n)_{n\in\mathds N}\subset\mathds L^\Phi(\mu)$ is a sequence which converges to $f$ in $(\mathds{L}^\Phi(\mu), \|\cdot\|_\Phi)$, then $f_n$ converges in measure (w.r.t.\ $\mu$) to $f$. In particular, there exists a sequence $\{f_{n(j)}\}_{j\in \mathds{N}}$ such that $f_{n(j)}$ converges to $f$ almost surely w.r.t.\ $\mu$;

    \item
        $C_c^\infty(\R^d)\subset \mathds{L}^\Phi(\mu)$ and $C_c^\infty(\R^d)$ is dense in $(\mathds{L}^\Phi(\mu),
        \|\cdot\|_\Phi)$.
    \end{enumerate}
 \end{lemma}

\begin{proof}
    Statements (1) and (2) are a consequence of \cite[Chapter 1, Theorem 1.7; Chapter 4, Theorem 4.7]{BS11}. We only need prove assertion (3). Let $f\in C_c^\infty(\R^d)$ and set $K=\supp(f)$. Then,
    $$
        \int \Phi(f)\,d\mu
        =\int_K \Phi(f)\,d\mu
        \le \mu(K)\,\sup_{x\in K}\Phi(f(x))
        <\infty,
    $$
    which means that $f\in\mathds{L}^\Phi(\mu)$.

    For any $f\in \mathds{L}^\Phi(\mu)$, define $f_n=\big[(f\wedge n)\vee(-n)\big]\,\I_{\{|f|\le n\}}$ for $n\ge1$. Since $f_n$ is a bounded function with compact support, $f_n\in\mathds{L}^\Phi(\mu)$. By the strong Fatou property of the function norm $\|\cdot\|_{(\Phi)}$, see \cite[Section 3.3; Page 57]{R}, we get that $\|f_n\|_{(\Phi)}\uparrow \|f\|_{(\Phi)}$ as $n\to\infty$.

    On the other hand, for every $n\geq 1$ there exists a sequence of functions $\{f_{n,k}\}_{k\in \mathds{N}}\subset C_c^\infty(\R^d)$ such that $f_{n,k}$ converges almost surely to $f_n$ with respect to $\mu$. Without loss of generality we may assume that $\bigcup_{k=1}^\infty \supp(f_{n,k})\subset [-(n+1),n+1]$ as well as $\sup_{k\ge 0}\|f_{n,k}\|_\infty \le n+1$. Egorov's theorem tells us that for every $\delta>0$, there is a set $N\subset[-(n+1),n+1]$ such that $\mu(N)<\delta$ and $f_{n,k}$ converges to $f_n$ uniformly on $N^c$. Using again the fact that $\Phi$ is increasing, we find for all $\lambda>0$
    \begin{align*}
        &\int \Phi\left(\frac{f_{n,k}-f_n}{\lambda}\right)d\mu\\
        &=\int_{[-(n+1),n+1]} \Phi\left(\frac{f_{n,k}-f_n}{\lambda}\right)d\mu\\
        &=\int_{N}\Phi\left(\frac{f_{n,k}-f_n}{\lambda}\right)d\mu
        + \int_{[-(n+1),n+1]\setminus N} \Phi\left(\frac{f_{n,k}-f_n}{\lambda}\right)d\mu\\
        &\le \mu(N)\sup_{0\le x\le 2(n+1)}\Phi\left(\frac x\lambda\right)
        +\int_{[-(n+1),n+1]\setminus N} \Phi\left(\sup_{x\in[-(n+1),n+1]\setminus N}\frac{|f_{n,k}(x)-f_n(x)|}{\lambda}\right)d\mu\\
        &= \mu(N)\,\Phi\left(\frac{2(n+1)}{\lambda}\right)
        +\mu([-(n+1),n+1])\Phi\left(\sup_{x\in[-(n+1),n+1]\setminus N}\frac{|f_{n,k}(x)-f_n(x)|}{\lambda}\right) .
    \end{align*}
    Letting first $k\to \infty$ and then $\delta\to 0$, we get
    $$
        \int \Phi\bigg(\frac{f_{n,k}-f_n}{\lambda}\bigg)d\mu\to 0,
    $$
    which yields that
    $$
        \|f_{n,k}-f_n\|_{(\Phi)}\to 0\quad\text{as\ \ }k\to \infty.
    $$
    This proves (3).
\end{proof}

The following statement is an extension of the Minkowski inequality in Orlicz space.

\begin{lemma}
    Let $\mu$ be a probability measure on $\R^d$, and $f\in B_b(\R^{2d})$ with $f\ge0$. Then,
    \begin{equation}\label{cou22}
        \left\|\int f(\cdot,y)\mu(dy)\right\|_{\Phi}
        \le \int \|f(\cdot,y)\|_{\Phi}\,\mu(dy).
    \end{equation}
\end{lemma}

\begin{proof}
    Since $\mu$ is a probability measure, $B_b(\R^d)\subset \mathds{L}^\Phi(\mu)$, so that both sides of the inequality \eqref{cou22} are well defined for all $f\in B_b(\R^{2d})$. For any $g>0$ with $\mu(\Phi_c(g))\le 1$,
    \begin{align*}
        \int g(x)\mu(dx)\int f(x,y)\,\mu(dy)
        &=\int \mu(dy)\int g(x)f(x,y)\,\mu(dx)\\
        &\le\int \mu(dy)\sup_{g>0,\,\,\int \Phi_c(g)\,d\mu \le 1}\int g(x)f(x,y)\,\mu(dx)\\
        &=\int\mu(dy)\,\|f(\cdot,y)\|_{\Phi},
    \end{align*}
    which yields the required assertion by taking supremum with respect to $g$ on the left hand side of the inequality above.
\end{proof}

\bigskip

We can now prove Theorem \ref{ou1}.

\begin{proof}[Proof of Theorem \ref{ou1}]
   (1)\,\, Let $\chi_k\in C_c^\infty(\R^d)$ be a sequence of cut-off functions with $\I_{B_k(0)}\le \chi_k\le 1$, $k\in \mathds{N}$. By the monotone convergence theorem, $\sup_{k\in \mathds{N}}T_t\chi_k=T_t1$. Since $\{T_t\}_{t\ge0}$ has the $C_b$-Feller property, $T_t\chi_k$ and $T_t1$ are continuous. Dini's theorem shows that the limit $T_t\chi_k \to T_t1$ is locally uniform. In particular, for every compact set $L\subset\R^d$ and every $\varepsilon>0$, there exists a compact set $K\subset \R^d$, $K=K_{L,\varepsilon, t}$ such that for $\chi\in C_c^\infty(\R^d)$ with $\I_K\le \chi\le 1$,
    \begin{equation}\label{ou1p1}
        \|\I_LT_t(1-\chi)\|_\infty\le\frac{\varepsilon}{2}.
    \end{equation}

    Let $u\in B_b(\R^d)\cap\mathds{L}^{\Phi_t}(\mu_t)$. Since $C_c^\infty(\R^d)$ is dense in $\mathds{L}^{\Phi_t}(\mu_t)$, we can choose a sequence $\{u_j\}_{j\in \mathds{N}}\subset C_c^\infty(\R^d)$ such that $\lim_{j\to \infty}\|u_j-u\|_{\Phi_t}=0$. By condition \eqref{cou1},
    $$
        \|(T_tu_j-T_tu_k)\I_{K}\|_\infty
        \le C(K,t)\|u_j-u_k\|_{\Phi_t}.
    $$
    This shows that $\{T_tu_j\}_{j\in \mathds{N}}$ is a Cauchy sequence in $C(\R^d)$ under the locally uniform convergence whose limit $\lim_{j\to \infty}T_tu_j=T_tu$ is again a continuous function.

    Thus, \eqref{cou1} holds for $u\in B_b(\R^d)\cap\mathds{L}^{\Phi_t}(\mu_t)$. Taking $u=\I_{N}$ with $\mu_t(N)=0$, we see that for Lebesgue almost every $x$ the kernel $P_t(x,\cdot)$ is absolutely continuous with respect to $\mu_t$, i.e.\
    $$
        T_tu(x)=\int u(y)p_t(x,y)\,\mu_t(dy)
        \quad\text{for\ \ }u\in B_b(\R^d)
        \text{\ \ and Lebesgue a.e.\ \ }x.
    $$
    By Lemma \ref{cou1s} (2), we can pick a subsequence $\{u_{j(k)}\}_{k\in \mathds{N}}$ which converges $\mu_t$-almost everywhere to $u$. Then, the dominated convergence theorem shows that $\lim_{k\to \infty}T_tu_{j(k)}=T_tu$ almost everywhere and, by the continuity of the functions involved, even everywhere. This means that just constructed extension is unique and coincides with the usual extension of $\{T_t\}_{t\ge0}$ using the kernels $P_t(x,\cdot)$; in particular, $T_t:B_b(\R^d)\cap\mathds{L}^{\Phi_t}(\mu_t)\to  C(\R^d)$.

    If $u\in B_b(\R^d)$, we fix $\varepsilon>0$ and a compact set $L\subset\R^d$, and choose $\chi$ as above. Then, for $x, y\in L$, we get using \eqref{ou1p1},
    \begin{align*}
        |T_t&u(x)-T_tu(y)|\\
        &\le |T_t(u\chi)(x)-T_t(u\chi)(y)|+|T_t(u(1-\chi))(x)|+|T_t(u(1-\chi))(y)|\\
        &\le|T_t(u\chi)(x)-T_t(u\chi)(y)|+2\|u\|_\infty\|\I_LT_t(1-\chi)\|_\infty\\
        &\le |T_t(u\chi)(x)-T_t(u\chi)(y)|+ \varepsilon\,\|u\|_\infty.
    \end{align*}
    Since $u\chi\in B_b(\R^d)\cap\mathds{L}^{\Phi_t}(\mu_t)$, we have $T_t(u\chi)\in C(\R^d)$, therefore $T_tu\in C_b(\R^d).$ This establishes the strong Feller property.

    \medskip
     (2)\,\,Conversely, assume that $\{T_t\}_{t\ge0}$ enjoys the strong Feller property. By Theorem \ref{thma} (c), for every $t>0$ there exists a probability measure $\mu_t$ such that the family of measures $(P_t(x,dy))_{x\in\R^d}$ is locally uniformly absolutely continuous with respect to $\mu_t$. Therefore, for any compact set $K$, the family of density functions $(p_t(x,\cdot))_{x\in K}$ is uniformly integrable with respect to $\mu_t$, see e.g.\  \cite[Theorem 16.8]{Schmp}. Consequently, there exists an even, increasing and convex function $\Phi_t:\R\to  \R^+ $ with $\Phi_t(0)=0$ and $\frac{\Phi_t(x)}{x}\nearrow\infty$ as $x\nearrow\infty$ and such that
    $$
        \sup_{x\in K}\int \Phi_t\big(p_t(x,y)\big)\mu_t(dy)<\infty.
    $$

    Let $u\in C_c^\infty(\R^d)$ with $u\ge0$. Denote by $\langle \cdot, \cdot\rangle_{L^2(\mu_t)}$ the inner product in the space ${L^2(\mu_t)}$. By the H\"{o}lder inequality
    \begin{align*}
        \|\I_K T_tu\|_{L^\infty(\mu_t)}
        &=\sup_{\begin{subarray}{c} g\in L^1_+(\mu_t)\\ \|g\|_{L^1(\mu_t)}=1\end{subarray}} \big\langle \I_KT_tu, g\rangle_{L^2(\mu_t)}\\
        &=\sup_{\begin{subarray}{c} g\in L^1_+(\mu_t)\\ \|g\|_{L^1(\mu_t)}=1\end{subarray}} \big\langle u, T^*_t(\I_Kg)\rangle_{L^2(\mu_t)}\\
        &\le 2\, \|u\|_{(\Phi_{t,c})} \sup_{\begin{subarray}{c} g\in L^1_+(\mu_t)\\ \|g\|_{L^1(\mu_t)}=1\end{subarray}} \|T^*_t(\I_K g)\|_{(\Phi_t)},
    \end{align*}
    where $\Phi_{t,c}$ is the Legendre transform of $\Phi_t$, and $T_t^*$ is the (formal) dual of $T_t$, i.e.\
    $$
        T^*_t g(y)
        = \int g(x)p_t(x,y)\,\mu_t(dx)
        \quad\text{for\ \ }g\in B_b(\R^d)
        \text{\ \ and\ \ }g\ge 0.
    $$
    Because of \eqref{cou22},
    \begin{align*}
        \sup_{\begin{subarray}{c} g\in L^1_+(\mu_t)\\ \|g\|_{L^1(\mu_t)}=1\end{subarray}} \|T^*_t(\I_Kg)\|_{(\Phi_t)}
        &= \sup_{\begin{subarray}{c} g\in L^1_+(\mu_t)\\ \|g\|_{L^1(\mu_t)}=1\end{subarray}} \left\|\int \I_K(x)g(x)p_t(x,\cdot)\,\mu_t(dx)\right\|_{(\Phi_t)}\\
        &\le \sup_{\begin{subarray}{c} g\in L^1_+(\mu_t)\\ \|g\|_{L^1(\mu_t)}=1\end{subarray}} \left\|\int \I_K(x)g(x)p_t(x,\cdot)\,\mu_t(dx)\right\|_{\Phi_t}\\
        &\le  \sup_{\begin{subarray}{c} g\in L^1_+(\mu_t)\\ \|g\|_{L^1(\mu_t)}=1\end{subarray}} \int \I_K(x)g(x)\|p_t(x,\cdot)\|_{\Phi_t}\,\mu_t(dx)\\
        &= \sup_{x\in K}\|p_t(x,\cdot)\|_{\Phi_t}.
    \end{align*}
    Therefore,
    $$
        \|\I_K T_tu\|_{L^\infty(\mu_t)}
        \le 2\,\|u\|_{(\Phi_{t,c})} \sup_{x\in K}\|p_t(x,\cdot)\|_{\Phi_t}
        \le 4\,\sup_{x\in K}\|p_t(x,\cdot)\|_{\Phi_t} \|u\|_{\Phi_{t,c}} .
    $$

    \medskip
     (3)\,\,Let us show that under the conditions (1) or (2) the inequality \eqref{cou1} follows from \eqref{cou2}.

    \medskip
    Assume first that $\{T_t\}_{t\geq 0}$ has the strong Feller property and is weakly irreducible in the sense of (1) of the statement of Theorem \ref{ou1}. From the proof of Theorem \ref{thma}, $(c)\Rightarrow(a)$, we know that the measure $\mu_t$ is of the form
    $$
        \mu_t(\cdot)=\frac{\int w(x)P_t(x,\cdot)dx}{\int w(x)P_t(x,\R^d)dx}
    $$
    for some continuous function $w>0$ with $\int w(x)\,dx = 1$. We will prove that
    $$
        \|\I_KT_tu\|_{L^\infty(\mu_t)}=\|\I_KT_tu\|_{L^\infty(dx)}=\sup_{x\in K}T_tu(x).
    $$
    It is clear that
    $$
        \|\I_KT_tu\|_{L^\infty(\mu_t)}\le\sup_{x\in K}T_tu(x).
    $$
    Suppose that there exists a Borel set $N$ with $\mu_t(N)=0$ such that
    $$
        \|\I_KT_tu\|_{L^\infty(\mu_t)}
        =\sup_{x\in K\setminus N}T_tu(x)
        <\sup_{x\in K}T_tu(x).
    $$
    Since $T_tu$ is a continuous function, there is some $x_0\in K$ such that $T_tu(x_0)=\sup_{x\in K}T_tu(x)$. We conclude that there exists some $\delta>0$ such that $B(x_0,\delta)\subset N$; otherwise, for every $\eta>0$, $B(x_0,\eta)\cap (K\setminus N)\neq \emptyset$, and we could choose a sequence $\{x_n\}_{n\in \mathds{N}}\subset K\setminus N$ with $x_n\to  x_0$. Then, however, $\sup_{x\in K\setminus N}T_tu(x)\ge \sup_{n}T_tu(x_n)=u(x_0)$. Therefore, $\|\I_KT_tu\|_{L^\infty(\mu_t)}=\sup_{x\in K}T_tu(x)$ which is a contradiction. Thus, $\mu_t(B(x_0,\delta))=0$. Thus,
    $$
        \int w(x)P_t(x,B(x_0,\delta))\,dx=0.
    $$
    Due to the weak irreducibility and the strong Feller property of $\{T_t\}_{t\ge0}$, we obtain $\mu_t(B(x_0,\delta))>0$, which is impossible. This proves our claim. Using \eqref{cou2} we easily deduce \eqref{cou1}.

    \medskip
    Assume now that $\{T_t\}_{t\geq 0}$ has the strong Feller property and that for every $u\in C_c^\infty(\R^d)$, $T_tu$ converges locally uniformly to $u$ as $t\to 0$. From the proof of \eqref{cou2} and the proof of Theorem \ref{thma}, $(a)\Rightarrow (d)$, we know that for every $t>0$, there is some Young function $\Phi_t:\R\mapsto \R^+$ with $\lim_{x\to \infty}\frac{\Phi_t(x)}{x}=\infty$ such that for all compact sets $K\subset\R^d$ and $u\in C_c^\infty(\R^d)$,
    \begin{equation*}
        \|\I_KT_tu\|_{\textrm{L}^\infty(\mu)}
        \le C(K,t)\,\|u\|_{\Phi_t},
    \end{equation*}
    where for some continuous $w>0$ satisfying $\int w(x)\,dx = 1$
    $$
        \mu(\cdot):=\frac{\int_0^\infty e^{-t}dt\int w(x)P_t(x,\cdot)dx}{\int_0^\infty e^{-t}dt\int w(x)P_t(x,\R^d)dx}.
    $$
    We claim that
    $$
        \|\I_KT_tu\|_{L^\infty(\mu)}
        =\|\I_KT_tu\|_{L^\infty(dx)}
        =\sup_{x\in K}T_tf(x).
    $$
    As before, the inequality `$\leq$' is trivial. On the other hand, assume that there exists a Borel set $N$ with $\mu(N)=0$ such that
    $$
        \|\I_KT_tf\|_{L^\infty(\mu)}
        =\sup_{x\in K\setminus N}T_tf(x)
        <\sup_{x\in K}T_tf(x).
    $$
    With a similar argument as above we get that there exist $x_0\in\R^d$ and $\delta>0$ such that $\mu(B(x_0,\delta))=0$. Choose $h\in C_c^\infty(\R^d)$ such that $\I_{B(x_0, \delta/2)}\le h\le \I_{B(x_0, \delta)}$. Obviously, $\int h\,d\mu=0$. Therefore,
    $$
        \int_0^\infty e^{-t}dt\int f(x)T_th(x)\,dx=0.
    $$
    Since $T_th$ converges locally uniformly to $h$ as $t\to 0$, there is some $t_0>0$ such that for all $s\le  t_0$,
    $$
        \sup_{x\in B(x_0,\delta/2)} |T_sh(x)-h(x)|\le 1/2.
    $$
    As $h(x)=1$ for all $x\in B(x_0,\delta/2)$, $T_sh(x)\ge 1/2$ for all $x\in B(x_0,\delta/2)$ and $s\le t_0$, which implies that $\int h\,d\mu > 0$. This is a contradiction and our claim is established. Now \eqref{cou1} follows easily from \eqref{cou2}.
\end{proof}

\section{Feller semigroups and processes}\label{1}
A \emph{Feller semigroup} is a sub-Markovian semigroup which has the Feller property and which is on $C_\infty(\R^d)$ strongly continuous: $\lim_{t\to 0}\|T_t f - f\|_\infty =0$ for all $f\in C_\infty(\R^d)$. A \emph{($d$-dimensional) Feller process} is a Markov process $\{X_t\}_{t\ge0}$ with state space $\R^d$ where the associated semigroup $\{T_t\}_{t\geq 0}$ is a Feller semigroup. Feller processes always have c\`adl\`ag, i.e.\ right-continuous with finite left-hand limits, versions and it is a routine argument that the Feller (or $C_b$-Feller) property of the semigroup together with the Markov property entail the strong Markov property of the process. As usual,
$$
    T_tu(x)=\Ee^x(u(X_t)),
    \quad u\in C_\infty(\R^d),\; t\ge 0,\; x\in\R^d,
$$
and the infinitesimal generator $(A,D(A))$ (of the semigroup or the process) is given by the strong limit
$$
    Au=\lim_{t\to 0}\frac{T_tu-u}{t}
$$
on the set $D(A)\subset C_\infty(\R^d)$ of those $u\in C_\infty(\R^d)$ where this limit exists in norm sense. We will call the generator of a Feller semigroup (or Feller process) a \emph{Feller generator}.

Under the assumption that the test functions $C_c^\infty(\R^d)$ are contained in $D(A)$ \cite{cou}, see also \cite{JS} and \cite{jac-book}, proved that the generator $A$ restricted to $C_c^\infty(\R^d)$ is a pseudo-differential operator
\begin{equation}\label{sf3}
    Au(x)
    =-p(x,D)u(x)
    :=-\int_{\R^d}e^{ix\cdot\xi}p(x,\xi)\hat{u}(\xi)\,d\xi,
\end{equation}
where $\hat{u}(\xi)=(2\pi)^{-d}\int e^{-ix\xi}u(x)\,dx$ denotes the Fourier transform. The \emph{symbol} of the operator, $p:\R^d\times \R^d\to  \mathds{C}$ appearing in \eqref{sf3}, is locally bounded and has the L\'{e}vy-Khinchine representation
\begin{equation}\label{sf2}
    p(x,\xi)
    =\frac 12\,\xi\cdot a(x)\xi -ib(x)\cdot\xi + \int_{z\neq 0} \big(1-e^{iz\cdot\xi}+iz\cdot\xi\, \I_{\{|z|\le1\}} \big)\,\nu(x,dz),
\end{equation}
where for each $x\in\R^d$ the triplet $(a(x),b(x),\nu(x,dz))$ is the L\'{e}vy characteristics, i.e.\ $a(x):=(a_{ij}(x))_{d\times d}$ is a nonnegative definite matrix-valued function, $b(x):=(b_i(x))$ is a measurable function, and $\nu(x,dz)$ is a nonnegative, $\sigma$-finite kernel on $\R^d\setminus\{0\}$ such that $\int_{z\neq 0} (1\wedge |z|^2)\nu(x,dz)<+\infty$ for every $x\in\R^d$.

As in Sections \ref{intro} and \ref{submark} we denote by $P_t(x,dy) = \Pp^x(X_t\in dy)$ the transition function of the process $X_t$, so that
\begin{equation*}
    T_tu(x)=\int_{\R^d}u(y)P_t(x,dy)
\end{equation*}
for all bounded and measurable functions $u:\R^d\to\R$.


\subsection{Feller generators  with bounded coefficients}

Since $\xi\mapsto p(x,\xi)$ has a L\'evy-Khinchine representation, it is not hard to see that
$$
    |p(x,\xi)| \leq h(x)\cdot (1+|\xi|^2)
    \quad\text{for all\ \ } x,\xi\in\R^d.
$$
In fact, $h(x)$ can be chosen to be $2\,\sup_{|\xi|\leq 1}|p(x,\xi)|$, cf.\ \cite{S3} or \cite{BS}. If $h$ is bounded, we say that the operator $p(x,D)$ has \emph{bounded coefficients} in the sense that $h$ is bounded if and only if
\begin{equation*}\label{bounded-coeff}
    \sup_{x\in\R^d}\max_{j,k}|a_{jk}(x)|
    + \sup_{x\in\R^d}\max_{j}|b_{j}(x)|
    + \sup_{x\in\R^d}\int_{z\neq 0} (1\wedge |z|^2)\,\nu(x,dz)
    < \infty,
\end{equation*}
see \cite{S3}.

If $C_c^\infty(\R^d)\subset D(A)$ and if $h$ is bounded, it is known that the continuity of $x\mapsto p(x,0)$ entails that $(T_t)_{t\ge0}$ is a $C_b$-Feller semigroup, cf.\ \cite[Theorem 4.3]{S1}. Thus, as a direct consequence of Theorems \ref{thma} and \ref{ou1}, we have

\begin{theorem}\label{thm2}
    Let $\{T_t\}_{t\ge0}$ be a Feller semigroup with generator $(A,D(A))$ such that $C_c^\infty(\R^d)\subset D(A)$. Suppose that $\sup_{x\in \R^d}|p(x,\xi)|\le c(1+|\xi|^2)$ for every $\xi\in \R^d$ and that the function $x\mapsto p(x,0)$ is continuous. The semigroup $(T_t)_{t\ge0}$ has the strong Feller property if one of the following conditions is satisfied:
    \begin{enumerate}
    \item
        For every $t>0$, the family of measures $(P_t(x,dy))_{x\in\R^d}$ is locally uniformly absolutely continuous with respect to Lebesgue measure, i.e.\ for any compact set $K\subset\R^d$
        $$
            \lim_{\varepsilon\to 0}\,\sup_{A\in {\Bb(\R^d)},\, \mathrm{Leb}(A)\le \varepsilon}\,\sup_{z\in K}P_t(z,A)=0.
        $$

    \item
        For every $t>0$, there exists some Young function $\Phi_t:\R\to  \R^+$ such that $\Phi_t(x)=0$ if and only if $x=0$ and such that for all compact sets $K\subset\R^d$ and all $u\in C_c^\infty(\R^d)$,
       \begin{equation*}
            \|\I_KT_tu\|_\infty\le C(K,t)\|u\|_{\Phi_t},
       \end{equation*}
       where $\|\cdot\|_{\Phi_t}$ denotes the norm of Orlicz space $\mathds{L}^{\Phi_t}(\mathrm{Leb})$;
    \end{enumerate}
\end{theorem}

If the Feller generator $-p(x,D)$ does not depend on $x$, i.e.\ if we have constant coefficients, the corresponding Feller processes are exactly the L\'evy processes---that is stochastic processes with independent and stationary increments---and the associated semigroups $\{T_t\}_{t\geq 0}$ are semigroups of convolution operators given by
$$
    T_tu(x)=\int_{\R^d}u(x+y)\mu_t(dy).
$$
Note that $\mu_t(\cdot-x) = P_t(x,\cdot) = \Pp^x(X_t\in\cdot) = \Pp^0(X_t+x\in\cdot)$ and that $\{\mu_t\}_{t\ge0}$ is a convolution semigroup of sub-probability measures on $\R^d$.

Hawkes, \cite[Theorem 2.2]{HAW}, proved that the semigroup corresponding to a L\'{e}vy process has the strong Feller property if and only if for each $t>0$ and $x\in\R^d$ the transition function $P_t(x,\cdot)=\mu_t(\cdot - x)$ is absolutely continuous with respect to Lebesgue measure. Therefore, for every $t>0$ and $x\in\R^d$, see e.g.\ \cite[Lemma 2.1]{C},
$$
    \lim_{\varepsilon\to 0}\sup\limits_{A\in {\Bb(\R^d)},\, \mathrm{Leb}(A)\le \varepsilon}P_t(x,A)
    = \lim_{\varepsilon\to 0}\sup\limits_{A\in {\Bb(\R^d)},\, \mathrm{Leb}(A)\le \varepsilon}\mu_t(A-x)
    =0.
$$
Since Lebesgue measure is invariant under translations, we conclude that L\'{e}vy processes have the strong Feller property if and only if for each $t>0$,
$$
    \lim_{\varepsilon\to 0}\sup\limits_{A\in {\Bb(\R^d)},\, \mathrm{Leb}(A)\le \varepsilon}\,\sup_{x\in\R^d} P_t(x,A)
    = \lim_{\varepsilon\to 0}\sup\limits_{A\in {\Bb(\R^d)},\, \mathrm{Leb}(A)\le \varepsilon}\mu_t(A)=0,
$$
which indicates that Theorem \ref{thm2} is sharp for L\'{e}vy processes. Based on this observation and the proof of Theorem \ref{ou1}, we obtain a new characterization of the strong Feller property of a L\'{e}vy process.

\begin{corollary}\label{newfeller}
    Let $\{X_t\}_{t\ge0}$ be a L\'{e}vy process on $\R^d$. Then the following two statements are equivalent:
    \begin{enumerate}
    \item
        The process $\{X_t\}_{t\ge0}$ has the strong Feller property;

    \item
        For every $t>0$, there exists some Young function $\Phi_t:\R\to  \R^+$ such that $\Phi_t(x)=0$ if and only if $x=0$ and such that for all $u\in C_c^\infty(\R^d)$,
        \begin{equation*}
            \|T_tu\|_\infty\le C(t)\|u\|_{\Phi_t},
        \end{equation*}
        where $\|\cdot\|_{\Phi_t}$ denotes the norm of Orlicz space $\mathds{L}^{\Phi_t}(\mathrm{Leb})$.
    \end{enumerate}
\end{corollary}

Let us finally apply Corollary
\ref{thma1} to stable-like processes. Loosely speaking, a
stable-like process on $\R^d$ is a Feller process, whose generator
has the same form as that of a rotationally symmetric stable L\'{e}vy
motion, but the index of stability depends on the position, e.g.\ see
\cite{BAss}. The associated generator is given by
$$L^{(\alpha)}u(x)=\int_{\R^d}\bigg(u(x+\xi)-u(x)-\langle\nabla u(x),z\rangle\I_{\{|z|\le1\}}\bigg)\,\frac{C_{\alpha(x)}}{|\xi|^{d+\alpha(x)}}\,d\xi,\qquad u\in C_b^2(\R^d),$$ where $0<\alpha(x)<2$, and $C_{\alpha(x)}$ is a constant defined through the L\'{e}vy-Khintchine formula
$$|\xi|^{\alpha(x)}=C_{\alpha(x)}\int_{\R^d\backslash\{0\}}\Big(1-\cos \langle \xi, z\rangle\Big)\,\frac{dz}{|z|^{d+\alpha(x)}},$$ i.e.\
 $$C_{\alpha(x)}=\alpha(x)2^{\alpha(x)-1}\Gamma\big((\alpha(x)+d)/2\big)\Big/\Big(\pi^{d/2}\Gamma\big(1-\alpha(x)/2\big)\Big),$$ see \cite[Exercise 18.23, Page 184]{BF}. In other words, the operator $L^{(\alpha)}$ can be regarded as a pseudo-differential operator of variable order with symbol $|\xi|^{\alpha(x)}$, i.e.\ $L^{(\alpha)}=-(-\triangle)^{\alpha(x)/2}$.

\begin{theorem}\label{strong} Assume that $\alpha\in C_b^\infty(\R^d)$ such that $0< \underline\alpha = \inf\alpha \leq\sup\alpha = \overline\alpha < 2$.
Then, there exists a unique Markov semigroup $\{P_t\}_{t\ge 0}$
associated with the symbol $|\xi|^{\alpha(x)}$, such that $\{P_t\}_{t>
0}$ is strong Feller. \end{theorem}
\begin{proof} According to \cite[Corollary 2.3]{BAss}, there exists a unique strong Markov process $(X_t,\Pp^x)$ for which $\Pp^x$ solves the martingale problem for $L^{(\alpha)}$ at each point $x\in\R^d$. For any $t\ge0$, $x\in\R^d$ and $f\in B_b(\R^d)$, we define
 $$P_tf(x)=\Ee^x(f(X_t)).$$ By \cite[Propositions 6.1 and 6.2]{BAss}, we know that $\{P_t\}_{t\ge0}$ is a Markov semigroup and has the $C_b(\R^d)$-Feller property.

Let $P_t(x,dy)$ be the kernel representing $P_t$, i.e.\ $P_t(x,A)=P_t\I_A(x)$ for all $t>0$, $x\in\R^d$ and $A\in \mathscr{B}(\R^d)$. \cite[Theorem 1.7]{Ne} furthermore shows that for any $t>0$ and $x\in\R^d$, $P_t(x,dy)$ has a density with respect to Lebesgue measure. Denote this density by $p_t(x,y)$. Note that $x\mapsto C_{\alpha(x)}$ is a positive function of $C_b^\infty(\R^d)$. Then, as a consequence of \cite[Theorem 5.1 and its Corollary, Pages 759--760]{Kolo}, we get that for every $t>0$, $p_t(x,y)$ is locally bounded on $\R^d\times \R^d$.

With all the conclusions above, the required assertion immediately follows from Corollary \ref{thma1}. \end{proof}

\subsection{Regularity of Feller processes with the strong Feller property}
Unless otherwise indicated, $\{X_t\}_{t\ge0}$ is a $d$-dimensional Feller process with generator $(A,D(A))$ such that $C_c^\infty(\R^d)\subset D(A)$. Then  $A|_{C_c^\infty(\R^d)}=-p(x,D)$ is a pseudo-differential operator with locally bounded symbol $p(x,\xi)$ given by \eqref{sf2}.

We want to show that the strong Feller property is closely connected with the fact that the function $x\mapsto \Ee^xu(X_{\tau_U})$ is continuous on $U$, where $u\in B_b(\R^d)$, $U$ is an open set and $\tau_U$ is the first exit time from the set $U$, i.e.
$$
    \tau_U=\inf\{t\ge0\::\: X(t)\notin U\}.
$$
\begin{theorem}\label{popop}
    Suppose that $\{X_t\}_{t\ge0}$ is a Feller process which has also the strong Feller property. Then, for any open set $U$ and all bounded measurable function $u$, the function $x\mapsto \Ee^xu(X_{\tau_U})$ is continuous on $U$.
 \end{theorem}

\begin{proof}
    Our approach uses ideas of \cite[Volume II; Section 13.3, Theorem 13.1, page 30--31]{Dy} and \cite[Lemma 2.2]{BH}. We split the proof into three steps.

    \medskip
    \emph{Claim 1: for any $t>0$ and all bounded measurable functions $u$, the function $x\mapsto f_t(x):=\Ee^x(\I_{\{\tau_U>t\}}u(X_t))$ is continuous on $U$.}

    Denote by $(\theta_t)_{t\ge0}$ the family of canonical shifts on the path space. Then we have for all $s\in(0,t)$ that
    $$
        \I_{\{\tau_U>s\}}\, \theta_s\big(\I_{\{\tau_U>t-s\}}u(X_{t-s})\big)
        =\I_{\{\tau_U>t\}}u(X_t).
    $$
    By the Markov property we get
    \begin{align*}
        f_t(x)
        &=\Ee^x\Big(\I_{\{\tau_U>s\}}\theta_s\big(\I_{\{\tau_U>t-s\}}u(X_{t-s})\big)\Big)\\
        &=\Ee^xf_{t-s}(X_s)-\Ee^x\Big(\I_{\{\tau_U\le s\}}\theta_s \big(\I_{\{\tau_U>t-s\}}u(X_{t-s})\big)\Big).
    \end{align*}
    Because of the strong Feller continuity, $x\mapsto \Ee^xf_{t-s}(X_s)$ is continuous. The second term is bounded by $\|u\|_\infty \Pp^x(\tau_U\le s)$. As $s\downarrow0$, it converges to zero uniformly on every compact subset of $U$, see Proposition \ref{maxx} below. This proves that the function $f_t$ is continuous on $U$.

    \medskip
    \emph{Claim 2: for any $t>0$ and all bounded measurable functions $u$, the function $g_t(x):=\Ee^x(u_{\tau_U\wedge t})$ is continuous on $U$.}

    It is enough to show that the function $h_t(x):=g_t(x)-f_t(x)=\Ee^x(\I_{\{\tau_U\le t\}}u(X_{\tau_U}))$ is continuous. As in the proof of Claim 1, we get for any $s\in(0,t)$,
    $$
        \I_{\{\tau_U>s\}}\, \theta_s(\I_{\{\tau_U\le t-s\}}u(X_{\tau_U}))
        = \I_{\{s<\tau_U\le t\}}u(X_{\tau_U}).
    $$
    Using the strong Markov property we find
    \begin{align*}
        h_t(x)
        &=\Ee^x\Big(\I_{\{\tau_U>s\}}\theta_s\big(\I_{\{\tau_U\le t-s\}} u(X_{\tau_U})\big)\Big)
          +\Ee^x\big(\I_{\{\tau_U\le s\}}u(X_{\tau_U})\big)\\
        &=\Ee^xh_{t-s}(X_s)
          + \Ee^x\bigg(\I_{\{\tau_U\le s\}}\Big[u(X_{\tau_U})-\theta_s\big(\I_{\{\tau_U\le t-s\}}u(X_{\tau_U})\big)\Big]\bigg).
    \end{align*}
    In view of the strong Feller property of the process $X_t$, the first term on the right, $x\mapsto \Ee^xh_{t-s}(X_s)$, is continuous. The second term is bounded by $2\,\|u\|_\infty \Pp^x(\tau_U\le s)$; letting $s\to 0$, it tends to zero uniformly on each compact subset of $U$, cf.\ Proposition \ref{maxx} below.

    \medskip
    Finally, if $u\in B_b(\R^d)$, we define $\tilde{u}$ by
    $$
        \tilde u(x) = c\,\I_{U}(x) + u(x)\I_{U^c}(x)
    $$
    where $c\le\inf_{x\in U^c}u(x)$. By Claim 2, the function $\Ee^x\tilde{u}(X_{\tau_U\wedge t})$ is continuous on $U$ for every $t>0$. Letting $t\uparrow\infty$, one has $\Ee^x\tilde{u}(X_{\tau_U\wedge t})\uparrow \Ee^xu(X_{\tau_U})$. Indeed, if $s,t>0$ we see by the very definition of the constant $c$ that
    \begin{align*}
        \Ee^x \tilde u(X_{\tau_U \wedge (s+t)})
        &= c\,\Pp^x(\tau_U > s+t) + \Ee^x\big(u(X_{\tau_U})\I_{\{\tau_U \leq s+t\}}\big)\\
        &= c\,\Pp^x(\tau_U > s+t)
            + \Ee^x\big(u(X_{\tau_U})\I_{\{t < \tau_U \leq s+t\}}\big)
            + \Ee^x\big(u(X_{\tau_U})\I_{\{\tau_u \leq t\}}\big)\\
        &\geq c\,\Pp^x(\tau_u > s+t) + c\,\Pp(t < \tau_U \leq s+t)
            + \Ee^x\big(u(X_{\tau_U})\I_{\{\tau_u \leq t\}}\big)\\
        &= \Ee^x \tilde u(X_{\tau_U \wedge t}).
    \end{align*}
    Hence, $\Ee^xu(X_{\tau_U})$ is lower semicontinuous. The same argument applied to $-u$ shows that $-\Ee^xu(X_{\tau_U})$ is lower semicontinuous. Therefore, the function $\Ee^xu(X_{\tau_U})$ is continuous. This finishes the proof.
\end{proof}

The following proposition has been used in the proof of Theorem \ref{popop} above.
\begin{proposition}\label{maxx}
    For any open set $U$, as $s\to 0$, the function $\Pp^x(\tau_U\le s)$ converges to zero uniformly on every compact set $D\subset U$.
\end{proposition}
\begin{proof}
    According to \cite[Theorem 1.1]{WJ}--- see also \eqref{1u1} below---for any $x\in \R^d$, $r>0$ and $t>0$,
    \begin{equation*}
        \Pp^x(\tau_{B(x,r)}\le t)\le c\,t \sup_{y \,:\, |y-x|\le r}\sup_{|\xi|\le 1/r}|p(y,\xi)|
    \end{equation*}
    with an absolute constant $c>0$. Set $r=\text{dist}(D, \partial U)>0$. Then,
    \begin{align*}
        \sup_{x\in D}\Pp^x(\tau_U\le t)
        &\le \sup_{x\in D}\Pp^x(\tau_{B(x,r)}\le t)\\
        &\le ct\sup_{x\in D}\sup_{|y-x|\le r}\sup_{|\xi|\le 1/r}|p(y,\xi)|\\
        &\le ct\sup_{x\in U}\sup_{|\xi|\le 1/r}|p(x,\xi)|.
    \end{align*}
    The assertion follows from the fact that the symbol $p(x,\xi)$ is locally bounded.
\end{proof}

A close inspection of the proof of Theorem \ref{popop} reveals that for all strong Markov processes which have the strong Feller property and which satisfy Proposition \ref{maxx} the function $x\mapsto \Ee^xu(X_{\tau_U})$ is continuous on $U$. There are many examples of such processes, e.g.\ all diffusion processes with strictly positive definite diffusion matrix in \cite{BH} and Feller processes which have the strong Feller property and whose generator has locally bounded coefficients.

\medskip

Let us indicate an interesting consequence of the continuity of $x\mapsto \Ee^xu(X_{\tau_U})$ for any $u\in B_b(\R^d)$ and any open set $U$.
\begin{theorem}\label{resov}
    Let $\{X_t\}_{t\geq 0}$ be a Feller process with transition semigroup $\{T_t\}_{t\geq 0}$ and generator $(A,D(A))$ given by \eqref{sf3}, \eqref{sf2}. Assume that for any bounded open set $U$ and all bounded measurable functions $u$, the function $x\mapsto \Ee^xu(X_{\tau_U})$ is continuous on $U$, where $\tau_U=\inf\{t\ge0:X(t)\notin U\}$. If
    $$
        \lim_{r\to 0}\int_{\{z\::\:|z|\ge r\}}\nu(x,dz) = \infty
        \quad\text{for all\ \ }x\in\R^d,
    $$
    then the corresponding resolvent operators $R_\alpha=\int _0^\infty e^{-\alpha t}T_tdt$, for every $\alpha>0$, have the strong Feller property, i.e.\ each $R_\alpha$ maps $B_b(\R^d)$ to $C_b(\R^d)$.
\end{theorem}

\begin{proof}
    Recall that a bounded measurable function $h:\R^d\to\R$ is said to be harmonic in $D\subset\R^d$ if for any bounded open set $U\subset \bar{U}\subset D$,
    $$
        h(x)=\Ee^x(h(X_{\tau_U})),\quad x\in U.
    $$
    This means that under the assumptions of Theorem \ref{resov}, all bounded harmonic functions are continuous.

    Next, we follow the proof of \cite[Theorem 6.3]{STD}. For fixed $x\in\R^d$, $\alpha, r>0$ and $u\in \Bb_b(\R^d)$, the strong Markov property shows for all $z\in B(x,r)$
    \begin{align*}
        R_\alpha u(z)
        &= \Ee^z\left[\int_0^{\tau_{B(x,r)}}e^{-\alpha t}u(X_s)\,ds\right]
         + \Ee^z\bigg[e^{-\alpha\tau_{B(x,r)}}R_\alpha u(X_{\tau_{B(x,r)}})\bigg]\\
        &=\Ee^z\left[\int_0^{\tau_{B(x,r)}}e^{-\alpha t} u(X_s)\,ds\right]\\
        &\quad
         +\Ee^z\bigg[\big(e^{-\alpha\tau_{B(x,r)}}-1\big)R_\alpha u\big(X_{\tau_{B(x,r)}}\big)\bigg]+ \Ee^z R_\alpha u\big(X_{\tau_{B(x,r)}}\big).
    \end{align*}
    Note that, the function $z\mapsto \Ee^z\big(R_\alpha u(X_{\tau_{B(x,r)}})\big)$ is bounded and harmonic in $B(x,r)$, hence it is continuous. Using the elementary estimates $|e^{-\alpha s}-1|\le \alpha s$ and $|e^{-\alpha s}|\le 1$ we get that the first two terms on the right hand of the inequality above do not exceed
    $$
        2\,\|u\|_\infty \sup_{z\in B(x,r)}\Ee^z\tau_{B(x,r)}
        + 2\alpha\,\|R_\alpha u\|_\infty \sup_{z\in B(x,r)}\Ee^z\tau_{B(x,r)}
        \le
        4\,\|u\|_\infty \sup_{z\in B(x,r)}\Ee^z\tau_{B(x,r)}.
    $$
    Using Proposition \ref{1u} below we find for any $y\in B(x,r)$,
    \begin{align*}
        \sup_{z\in B(x,r)}\Ee^z\tau_{B(x,r)}
        \le \sup_{z\in B(x,r)}\Ee^z\tau_{B(z,2r)}
        &\le 2 \left[{\inf_{y\,:\,|y-x|<4r}\int_{\{z\::\:|z|\ge 9r\}} \nu(y,dz)} \right]^{-1}\\
        &\le 2\left[{\int_{\{z\::\:|z|\ge 9r\}}\nu(x,dz)} \right]^{-1} .
    \qedhere
    \end{align*}
\end{proof}

The proposition below provides the estimates for the first exit time, which are interesting on their own. For L\'{e}vy processes \eqref{1u1} can be found in \cite[(3.1)]{P}.

\begin{proposition}\label{1u}
    Let $\{X_t\}_{t\ge0}$ be a Feller process with generator $(A,D(A))$ such that $C_c^\infty(\R^d)\subset D(A)$ and $A|_{C_c^\infty(\R^d)}=-p(x,D)$ with symbol $p(x,\xi)$ given by \eqref{sf2}. For any $x\in\R^d$, $1>r>0$ and $t>0$,
    \begin{equation}\label{1u1}
        \Pp^x( \tau_{B(x,r)}\le t)\le \frac{t\,H(x,r)}{1-e^{-1}}
    \end{equation}
    and
    \begin{equation}\label{1u2}
        \Pp^x( \tau_{B(x,r)}> t)\le \frac{1}{t\,h(x,r)},
    \end{equation}
    hold with
    \begin{align*}
        H(x,r)
        &=\sup_{y\,:\,|y-x|\le r} \Bigg\{\frac{2}{r^2}\left[\frac 12\sum_{j=1}^d a_{jj}(y)+ (y-x)\cdot\left(b(y)-\int_{\{r <|z|\le 1\}}z\,\nu(y,dz)\right)\right]\\
        &\phantom{\sup_{y\,:\,|y-x|\le r}\quad(y}
        +\frac{1}{r^2}\int_{\{|z|\le r\}}|z|^2\,\nu(y,dz)+\int_{\{|z|> r\}}\,\nu(y,dz) \Bigg\}
    \end{align*}
    and
    $$
        h(x,r)=\inf_{y\,:\,|y-x|<r}\int_{\{z:|z|\ge 3r\}}\nu(y,dz).
    $$
    Moreover, we have
    $$
        \frac{c_1}{H(x,t)}
        \le \Ee^x\tau_{B(x,2r)}
        \le \frac{c_2}{h(x,3r)},
    $$
    where $c_1$ and $c_2$ are some positive absolute constants.
\end{proposition}

\begin{proof}
    The L\'evy-Khinchine formula \eqref{sf2} and a standard Fourier inversion argument in \eqref{sf3} show that the generator $A$ has the following integro-differential representation
    \begin{align*}
        Au(x)
        = &\frac 12 \sum_{j,k=1}^d a_{jk}(x)\partial_{jk}u(x)
        + b(x)\cdot\nabla u(x)\\
        &+\int_{z\neq 0} \big(u(x+z)-u(x)-\nabla u(x)\cdot z\,\I_{\{|z|\le1\}}\big)\,\nu(x,dz);
    \end{align*}
    note that this expression is well defined for all $u\in C_b^2(\R^d)$, cf.\ e.g.\ \cite{S3}. For any $x\in\R^d$ and $r>0$, set $u_r^x(y)= \exp\big(\!\!-|y-x|^2/r^2\big)$. Then,
    \begin{align*}
        A&u_r^x(y)\\
        &=u_r^x(y) \Bigg\{\frac{2}{r^2}\Bigg[\frac{1}{r^2}\sum_{j,k=1}^da_{jk}(y)(y_j-x_j)(y_k-x_k) - \frac 12 \sum_{j=1}^da_{jj}(y)\\
        &\quad\quad\quad\quad\qquad \mbox{}-(y-x)\cdot \left(b(y)-\int_{\{r <|z|\le 1\}}z\,\nu(y,dz)\right)\Bigg]\\
        &\quad\quad\qquad\mbox{}+\int_{\{|z|\le r\}}\left[\exp\left({-\frac{|y-x+z|^2-|y-x|^2}{r^2}}\right)-1+\frac{2(y-x) \cdot z}{r^2}\right]\nu(y,dz)\\
        &\quad\quad\qquad\mbox{}+\int_{\{|z|>r\}} \left[ \exp\left( {-\frac{|y-x+z|^2-|y-x|^2}{r^2}}\right)-1\right]\nu(y,dz)\Bigg\}.
    \end{align*}
    Since $u^x_r\in\mathscr S$ is a rapidly decreasing function and, as such, in the domain of the Feller generator $A$---see e.g.\ \cite{S3}---standard stopping arguments show that
    $$
        M_t:=1-u^x_r(X_{t\wedge \tau_{B(x,r)}}) + \int_0^{{t\wedge \tau_{ B(x,r)}}}Au_r^x(X_s)ds
    $$
    is for every $\Pp^x$ a martingale w.r.t.\ the canonical filtration $\{\mathscr{F}_t\}_{t\geq 0}$ of the Feller process $\{X_t\}_{t\geq 0}$; moreover $\Ee^x(M_t)=0$ for all $t\ge 0$. Now
    \begin{align*}
        \Pp^x(\tau_{ B(x,r)}\le t)
        &\le (1-e^{-1})^{-1}\,\Ee^x\big(1-u_r^x(X_{t\wedge\tau_{B(x,r)}})\big)\\
        &\le (1-e^{-1})^{-1}\,\Ee^x(t\wedge \tau_{B(x,r)})\sup_{y\,:\,|y-x|\le r} \big(-Au_r^x(y)\big).
    \end{align*}
    Using the inequality $1-e^{-x}\le x$, $x>0$, we find
    \begin{align*}
        \sup_{y\,:\,|y-x|\le r}&\big(-Au_r^x(y)\big)\\
        &\le \sup_{y\,:\,|y-x|\le r}\Bigg\{\frac{2}{r^2}\Bigg[\frac 12 \sum_{j=1}^da_{jj}(y)+ (y-x)\cdot \left(b(y)-\int_{\{r <|z|\le1\}}z\,\nu(y,dz)\right)\\
        &\qquad\qquad\qquad\quad\mbox{} -\frac{1}{r^2}\sum_{j,k=1}^da_{jk}(y)(y_j-x_j)(y_k-x_k)\Bigg]\\
        &\qquad\qquad\qquad\mbox{} +\frac{1}{r^2}\int_{\{|z|\le r\}}|z|^2\,\nu(y,dz)
                           +\int_{\{|z|> r\}}\nu(y,dz) \Bigg\}\\
        &\le\sup_{y\,:\,|y-x|\le r}\Bigg\{\frac{2}{r^2}\Bigg[\frac 12 \sum_{j=1}^da_{jj}(y)+ (y-x)\cdot \left(b(y)-\int_{\{r <|z|\le 1\}}z\,\nu(y,dz)\right) \Bigg]\\
        &\quad\,\,\qquad\qquad\mbox{}+\frac{1}{r^2}\int_{\{|z|\le r\}}|z|^2\,\nu(y,dz)+\int_{\{|z|> r\}}\nu(y,dz) \Bigg\}.
    \end{align*}
    This completes the proof of \eqref{1u1}.

    For \eqref{1u2} choose a function $v_r^x\in C_c^\infty(\R^d)$ such that $v_r^x(y)=1$ if $y\in B(x,r)$ and $v_r^x(y)=0$ if $y\notin B(x,2r)$. Then,
    \begin{align*}
        1 &- \Pp^x(\tau_{B(x,r)}> t)\\
        &\ge 1-\Ee^x\big(v_r^x(X_{t\wedge\tau_{ B(x,r)}})\big)\\
        &\ge \Ee^x\left(\int_0^{t\wedge\tau_{ B(x,r)}}-Lv_r^x(X_{s-})\,ds\right)\\
        &= \Ee^x\left(\int_0^{t\wedge\tau_{B(x,r)}}\left[\int\Big(v_r^x(y)-v_r^x(y+z)+\nabla v_r^x(y)\cdot z\I_{\{|z|\le 1\}}\Big)\nu(y,dz)\right]\bigg|_{y=X_{s-}}ds\right)\\
        &= \Ee^x\left(\int_0^{t\wedge\tau_{B(x,r)}}\left[\int\Big(1-v_r^x(y+z)\Big)\nu(y,dz)\right]\bigg|_{y=X_{s-}}ds\right)\\
        &\ge t\,\Pp^x(\tau_{ B(x,r)}> t)\,\inf_{y\in B(x,r)}\int_{\{z\::\:|y+z-x|\ge 2r\}}\nu(y,dz)\\
        &\ge t\,\Pp^x(\tau_{ B(x,r)}> t)\,\inf_{y\in B(x,r)}\int_{\{z\::\:|z|\ge 3r\}}\nu(y,dz).
    \end{align*}
    This proves \eqref{1u2}.

    Finally, for any $x\in\R^d$ and $r>0$, the Markov property and \eqref{1u2} entail
    \begin{align*}
        \Pp^x(\tau_{B(x,r/2)}>t)
        &= \Ee^x\Big(\Pp^{X_{t/2}}(\tau_{B(x,r/2)}>t/2)\cdot \I_{\{\tau_{B(x,r/2)}>t/2\}}\Big)\\
        &\le \Pp^x(\tau_{B(x,r/2)}>t/2)\,\sup_{y\in B(x,r/2)}\Pp^y(\tau_{B(y,r)}>t/2)\\
        &\le 4\Big(t^2\,h(x,r/2)\,\inf_{y\in B(x,r/2)}h(y,r)\Big)^{-1}\\
        &\le 4t^{-2}\,\big(h(x,3r/2)\big)^{-2}.
    \end{align*}
    This estimate and \eqref{1u1} at hand allow us to use the argument from \cite[Theorem 4.7]{S3} to get the two-sided estimate for $\Ee^x \tau_{B(x,2r)}$.
\end{proof}

We close this section with the remark that the argument of Theorem \ref{resov} also yields
\begin{corollary}\label{resov1}
    Let $\{X_t\}_{t\geq 0}$ be a Feller process with transition semigroup $\{T_t\}_{t\geq 0}$ and generator $(A,D(A))$ given by \eqref{sf3}, \eqref{sf2}. Assume that for any bounded open set $U$ and all bounded continuous functions $u\in C_b(\R^d)$, the function $x\mapsto \Ee^xu(X_{\tau_U})$ is continuous on $U$, where $\tau_U=\inf\{t\ge0:X(t)\notin U\}$. If
    $$
        \lim_{r\to 0}\int_{\{z\::\:|z|\ge r\}}\nu(x,dz) = \infty
        \quad\text{for all\ \ }x\in\R^d,
    $$
    then the corresponding resolvent operators $R_\alpha=\int _0^\infty e^{-\alpha t}T_tdt$, for every $\alpha>0$, have the $C_b$-Feller property, i.e.\ each $R_\alpha$ maps $C_b(\R^d)$ to $C_b(\R^d)$.
\end{corollary}

\bigskip

\noindent{\bf Acknowledgement.} Financial support through DFG (grant
Schi 419/5-1) and DAAD (PPP Kroatien) (for Ren\'{e} L.\ Schilling)
and the Alexander-von-Humboldt Foundation
  and the Natural Science Foundation of Fujian $($No.\ 2010J05002$)$
(for Jian Wang) is gratefully acknowledged.

\end{document}